\documentclass[11pt,reqno,a4paper]{amsart}

\usepackage[T1]{fontenc}
\usepackage[english]{babel}
\usepackage{microtype}
\usepackage{textcomp}
\usepackage{fbb}
\usepackage{amsmath, amssymb, amsthm, amscd, amsfonts}
\usepackage{mathtools}
\usepackage{mathrsfs}
\numberwithin{equation}{section}

\usepackage{graphicx}
\usepackage{tikz}
\usetikzlibrary{arrows.meta}
\usepackage{pgfplots}
\pgfplotsset{compat=1.15}

\usepackage{geometry}
\usepackage{fullpage}
\usepackage{setspace}
\usepackage{ragged2e}
\usepackage{wrapfig}
\usepackage{multicol}
\usepackage{float}
\usepackage{caption}
\usepackage{subcaption}
\usepackage{booktabs}
\usepackage{enumerate}
\usepackage{enumitem}
\setlist{
  listparindent=\parindent,
  parsep=0pt,
}

\definecolor{darkblue}{rgb}{0,0,0.7}
\definecolor{darkred}{rgb}{0.7,0,0}
\usepackage[
  colorlinks=true,
  linkcolor=darkred,
  citecolor=darkblue,
  urlcolor=blue,
  breaklinks=true
]{hyperref}
\usepackage{cleveref}

\newcommand{\reg}{\operatorname{reg}}

\newcommand{\supp}{\operatorname{Supp}}

\DeclareMathOperator{\indmatch}{indmatch}
\DeclareMathOperator{\aim}{aim}

\newcommand{\K}{\mathbb{K}}

\newtheorem{theorem}{Theorem}[section]
\newtheorem{definition}[theorem]{Definition}

\newtheorem{lemma}[theorem]{Lemma}
\newtheorem{proposition}[theorem]{Proposition}

\newtheorem{observation}[theorem]{Observation}

\newtheorem{conjecture}[theorem]{Conjecture}

\newtheorem{notation}[theorem]{Notation}
\newtheorem*{notation*}{Notation}
\newtheorem{discussion}[theorem]{Discussion}
\theoremstyle{definition}
\newtheorem*{convention}{Convention}

\begin{document}

\title[Regularity of Squarefree Powers of Edge Ideals of Whiskered Cycles]
{Regularity of Squarefree Powers of Edge Ideals of Whiskered Cycles}

\author{Sanjoy Das}
\address{Department of Mathematics, Indian Institute Of Technology  Bhubaneswar, Bhubaneswar, 752050, India}
\email{a25ma09004@iitbbs.ac.in}

\author{Arka Ghosh}
\address{Department of Mathematics, Indian Institute Of Technology  Bhubaneswar, Bhubaneswar, 752050, India}
\email{arkaghosh1208@gmail.com}

\author{S Selvaraja}
\address{Department of Mathematics, Indian Institute Of Technology  Bhubaneswar, Bhubaneswar, 752050, India}
\email{selvas@iitbbs.ac.in}

\thanks{AMS Classification 2020: 13D02,  05E40, 05C70, 13F55.}

\begin{abstract}
Let $G$ be a finite simple graph and let $I(G)$ denote its edge ideal. 
For $q \ge 1$, the $q$-th squarefree power $I(G)^{[q]}$ is generated by squarefree monomials corresponding to matchings of size $q$ in $G$. 
We denote by $\operatorname{reg}(-)$ the Castelnuovo-Mumford regularity. 
Das, Roy, and Saha~\cite{DRS24} conjectured that if $G = W(C_n)$ is a whiskered cycle, then
\[
\operatorname{reg}\big(I(G)^{[q]}\big)
= 2q + \left\lfloor \frac{n - q - 1}{2} \right\rfloor
~ \text{for all } 1 \le q \le \nu(G),
\]
where $\nu(G)$ denotes the matching number of $G$. 
In this paper, we confirm this conjecture by determining the exact value of $\operatorname{reg}(I(G)^{[q]})$.
\end{abstract}

\maketitle

\section{Introduction}

Let $R = \K[x_1,\dots,x_n]$ be a polynomial ring over a field $\K$, and let $I \subseteq R$ be a homogeneous ideal. 
Consider a minimal graded free resolution of $I$:
\[
0 \longrightarrow \bigoplus_{j} R(-j)^{\beta_{p,j}(I)} 
\longrightarrow \cdots \longrightarrow \bigoplus_{j} R(-j)^{\beta_{0,j}(I)} 
\longrightarrow I \longrightarrow 0,
\]
where $p \le n$. The \emph{Castelnuovo-Mumford regularity} (or simply the regularity) of $I$ is defined by
$\reg(I) = \max\{\, j-i \mid \beta_{i,j}(I) \neq 0 \,\}.$
This invariant measures the complexity of the minimal graded free resolution of $I$ and plays a fundamental role in commutative algebra.
There is a well-known correspondence between quadratic squarefree monomial ideals of $R$ and graphs (finite, with no loops and no multiple edges) on the vertex set $\{x_1,\dots,x_n\}$. 
Given a graph $G$ with edge set $E(G)$, its \emph{edge ideal} is defined by
$I(G) = (x_i x_j \mid \{x_i,x_j\} \in E(G)) \subseteq R.$
The regularity of edge ideals and their powers has been extensively studied (see, for example,~\cite{BBH17, Ha2}). 
A fundamental result due to Cutkosky-Herzog-Trung~\cite{CHT} and Kodiyalam~\cite{vijay} asserts that, for any homogeneous ideal $I$, the function $\reg(I^q)$ is eventually linear in $q$ for $q \gg 0$. 
In the case of edge ideals, sharper bounds are known. More precisely, for any graph $G$ and all integers $q \ge 1$,
\[
2q + \indmatch(G) - 1 \le \reg(I(G)^q) \le 2q + \nu(G) - 1,
\]
where $\indmatch(G)$ and $\nu(G)$ denote the induced matching number and the matching number of $G$, respectively; see~\cite{selvi_ha, JS21}.

Squarefree powers of edge ideals have recently attracted considerable attention. 
Let $G$ be a finite simple graph and let $I(G) \subseteq R$ be its edge ideal. 
For an integer $q \ge 1$, the \emph{$q$-th squarefree power} of $I(G)$ is defined by
\[
I(G)^{[q]}
=
\Big(
\prod_{v \in \bigcup_{i=1}^q e_i} v
\;\Big|\;
e_1,\dots,e_q \in E(G),\;
e_i \cap e_j = \emptyset \ \text{for } i \neq j
\Big).
\]
Equivalently, $I(G)^{[q]}$ is generated by squarefree monomials corresponding to matchings of size $q$ in $G$. 
In particular, $1 \le q \le \nu(G)$, and $I(G)^{[q]} = (0)$ for $q > \nu(G)$. 
The study of squarefree powers was initiated in~\cite{BHZ18} and further developed in~\cite{EHHM22}; see also 
\cite{TKAK25, CF26, CFE25, DRS24, DRS25, EF26, EHHS24, EH21, EH25, FM24, RS26, HS25, Sey24, Fa25}.

A natural problem is to establish bounds for the regularity of squarefree powers analogous to those known for ordinary powers. 
Erey, Herzog, Hibi, and Madani~\cite[Theorem~2.1]{EHHM22} proved the lower bound
$q + \indmatch(G) \le \reg(I(G)^{[q]}).$
In the same work, they conjectured the upper bound $\reg(I(G)^{[q]}) \le q + \nu(G)$. 
In~\cite{BHZ18}, Bigdeli, Herzog, and Nahandi showed that for any graph $G$,
$\reg(I(G)^{[\nu(G)]}) = 2\nu(G).$
Fakhari~\cite{Sey24} verified the conjectured upper bound for several classes of graphs and later confirmed it in full generality in~\cite{Fa25}.
Erey and Hibi~\cite{EH21} introduced the notion of $q$-admissible matchings, denoted by $\aim^*(G,q)$, namely matchings that admit a partition into pairwise disjoint submatchings satisfying suitable gap conditions and whose induced subgraphs are forests. 
They proved that if $G$ is a forest, then
$\reg(I(G)^{[q]}) \le q + \aim^*(G,q),$
and conjectured that equality holds, that is,
$\reg(I(G)^{[q]}) = q + \aim^*(G,q).$
This conjecture was later confirmed by Crupi, Ficarra, and Lax~\cite{CFE25} using Betti splitting.
Furthermore, Chau, Das, Roy, and Saha~\cite{TKAK25} introduced the notion of generalized $q$-admissible matchings, denoted by $\aim(G,q)$. 
These invariants satisfy $\aim^*(G,q) \le \aim(G,q)$ for all graphs $G$, with equality when $G$ is a forest. 
They also showed that if $G$ is a block graph, then
$\reg(I(G)^{[q]}) = q + \aim(G,q).$
In~\cite{DRS24}, Das, Roy, and Saha determined the regularity of squarefree powers for several families of graphs, including paths, cycles, and whiskered paths, explicitly in terms of the number of vertices. 
In the same work, they proposed the following conjecture.

\begin{conjecture}\cite[Conjecture 6.2]{DRS24}
Let $G = W(C_n)$. Then, for $1 \le q \le n$,
\[
\reg\big(I(G)^{[q]}\big)
= 2q + \left\lfloor \frac{n - q - 1}{2} \right\rfloor.
\]
\end{conjecture}

They proved the lower bound for this conjecture and established the upper bound 
$\reg(I(G)^{[q]}) \le 2 + \left\lfloor \frac{n - q}{2} \right\rfloor$
(see~\cite[Corollary~5.3]{DRS24}). 
In this paper, we confirm the conjecture by sharpening the upper bound to match the lower bound (see Theorem~\ref{main}).

\section{Preliminaries}\label{pre}

In this section, we recall some basic definitions and notation used throughout the paper. 
All graphs considered in this work are finite and simple. 
Let $G$ be a graph with vertex set $V(G)$ and edge set $E(G)$. 
For a subset $U \subseteq V(G)$, the \emph{induced subgraph} $G \setminus U$ is the graph obtained by removing the vertices in $U$ together with all edges incident to them. 
In particular, for a vertex $x \in V(G)$, we write $G \setminus x$ for $G \setminus \{x\}$. For a subset $A \subseteq V(G)$, we define
$N_G(A) = \{y \in V(G) \mid \exists\, x \in A \text{ such that } \{x,y\} \in E(G)\},$
and
$N_G[A] = N_G(A) \cup A.$
A \emph{matching} in $G$ is a set of pairwise disjoint edges. 
The \emph{matching number} of $G$ is defined by
$\nu(G) = \max\{\, |M| \mid M \text{ is a matching of } G \,\}.$
An \emph{induced matching} in $G$ is a matching $M$ such that no two edges of $M$ are joined by an edge of $G$. 
The \emph{induced matching number} of $G$ is defined by
$\indmatch(G) = \max\{\, |M| \mid M \text{ is an induced matching of } G \,\}.$
A matching $M$ is called \emph{perfect} if every vertex of $G$ is incident to an edge of $M$. 
In particular, if $G$ admits a perfect matching and $|V(G)|=n$, then $n$ is even and $|M|=\frac{n}{2}$.
Let $G$ be a graph. A subset $A \subseteq V(G)$ is called \emph{independent} if no two vertices in $A$ are adjacent. 
The \emph{independence number} of $G$ is
$\alpha(G)=\max\{\,|A| \mid A \subseteq V(G)\ \text{is independent}\,\}.$

A \emph{path} $P_n$ is the graph with vertex set $\{x_1,\ldots,x_n\}$ and edge set $\{\{x_i,x_{i+1}\} \mid 1 \le i \le n-1\}$. 
We denote this path by $x_1x_2\cdots x_n$.
A \emph{cycle} $C_n$ is a graph on vertices $\{x_1,\ldots,x_n\}$ with edges $\{x_i,x_{i+1}\}$ for $1 \le i \le n-1$ together with $\{x_n,x_1\}$.
A graph is \emph{bipartite} if its vertex set can be partitioned into two disjoint sets such that no edge has both endpoints in the same set. 

A graph $G$ is said to be \emph{connected} if for every pair of vertices $x,y \in V(G)$, there exists a path in $G$ joining $x$ and $y$. 
A \emph{connected component} of $G$ is a maximal connected subgraph of $G$, that is, a connected subgraph that is not properly contained in any larger connected subgraph of $G$.

The \emph{whisker graph} of a graph $G$, denoted by $W(G)$, is the graph with vertex set 
$V(G) \cup \{y_i \mid x_i \in V(G)\}$ and edge set 
$E(G) \cup \{\{x_i,y_i\} \mid x_i \in V(G)\}$.
The \emph{complement} of a graph $G$, denoted by $\overline{G}$, is the graph on the same vertex set $V(G)$ in which $\{x,y\} \in E(\overline{G})$ if and only if $\{x,y\} \notin E(G)$ for distinct vertices $x,y$. 
A graph is called \emph{chordal} if every cycle of length at least four has a chord. 
A graph is called \emph{co-chordal} if its complement is chordal. 

Let $G$ and $H$ be graphs. An \emph{isomorphism} from $G$ to $H$ is a bijection
$\varphi : V(G) \longrightarrow V(H)$
such that $\{u,v\} \in E(G)$ if and only if $\{\varphi(u),\varphi(v)\} \in E(H)$. 
If such a map exists, we say that $G$ and $H$ are \emph{isomorphic}, and write $G \cong H$. Let $G_1,\ldots,G_k$ be graphs with pairwise disjoint vertex sets. 
The \emph{disjoint union} of $G_1,\ldots,G_k$, denoted by
$\coprod_{i=1}^k G_i,$
is the graph $G$ with vertex set
$V(G) = \bigcup_{i=1}^k V(G_i)$
and edge set
$E(G) = \bigcup_{i=1}^k E(G_i).$

We recall the notion of even-connections introduced by Banerjee~\cite{banerjee}, which plays a key role in the study of colon ideals of powers of edge ideals.
\begin{definition}[\cite{banerjee}]
Let $G$ be a graph and let $e_1\cdots e_q$ be an $q$-fold product of edges of $G$ (with repetition allowed). 
For $u,v\in V(G)$, we write
$u \sim_{e_1\cdots e_q} v$
if there exist an integer $r\ge 1$ and a sequence $p_0,\ldots,p_{2r+1}$ of vertices such that
\begin{enumerate}
\item $p_0=u$ and $p_{2r+1}=v$;
\item $\{p_k,p_{k+1}\}\in E(G)$ for all $0\le k\le 2r$;
\item $\{p_{2k+1},p_{2k+2}\}=e_i$ for some $i$, for each $0\le k\le r-1$;
\item for every $i$,
$\bigl|\{\,k \mid \{p_{2k+1},p_{2k+2}\}=e_i\,\}\bigr|
\le
\bigl|\{\,j \mid e_j=e_i\,\}\bigr|.$
\end{enumerate}
In this case, the sequence $p_0,\ldots,p_{2r+1}$ is called an \emph{even-connection} between $u$ and $v$ (with respect to $e_1\cdots e_q$).
\end{definition}

We define the support of a collection of edges and of a monomial as follows.

\begin{definition}
Let $G$ be a graph and let $M \subseteq E(G)$. We define
$\supp(M) := \{x \in V(G) \mid x \in e \text{ for some } e \in M\}.$
For a monomial $m \in R = \K[x_1,\ldots,x_n]$, we define
$\supp(m) := \{x_i \mid x_i \text{ divides } m\}.$
\end{definition}
We now define a graph associated to a matching via even-connections.
\begin{definition}
Let $G$ be a graph and let $M = \{e_1,\ldots,e_q\}$ be a matching of $G$ with $q \ge 1$. 
We define a graph $G^{M}$ as follows: $V(G^{M}) = V(G) \setminus \supp(M)$, and for $x_i, x_j \in V(G^{M})$,
\[
\{x_i, x_j\} \in E(G^{M})
~ \Longleftrightarrow ~
\{x_i, x_j\} \in E(G)
\ \text{or}\
x_i \sim_{e_1 \cdots e_q} x_j.
\]
\end{definition}
We denote by $\mathcal{G}(I)$ the set of minimal monomial generators of a monomial ideal $I$. The following result relates colon ideals of squarefree powers to the graph $G^{M}$ defined via even-connections.
\begin{theorem}[\cite{Sey24}, Corollary~3.4]\label{Thm-even}
Let $G$ be a graph and let $M=\{e_1,\ldots,e_q\}$ be a matching of $G$ with $q\ge 1$. 
Set $u=e_1\cdots e_q\in \mathcal{G}(I(G)^{[q]})$. 
Then every minimal generator of $(I(G)^{[q+1]}:u)$ has degree $2$. Moreover,
$(I(G)^{[q+1]}:u)=I(G^{M}).$
\end{theorem}

We recall the following known formula for the regularity of whisker graphs.

\begin{theorem}[{\cite[Theorem~1.3]{mohammad}, \cite[Lemma~21]{russ}}]\label{reg-known}
Let $G=W(H)$. Then
\[
\reg(I(G))=\alpha(H)+1=
\begin{cases}
2+\left\lfloor\frac{n-1}{2}\right\rfloor & \text{if } H=P_n,\\[4pt]
2+\left\lfloor\frac{n-2}{2}\right\rfloor & \text{if } H=C_n.
\end{cases}
\]
\end{theorem}
We collect some known results on the regularity of edge ideals that will be used throughout the paper.
\begin{theorem}
Let $G$ be a graph.
\begin{enumerate}
\item \cite[Lemma~8]{russ} If $G = G_1 \coprod G_2$, then
$\reg(I(G)) = \reg(I(G_1)) + \reg(I(G_2)) - 1.$

\item \cite[Lemma 3.1]{Ha2} If $H$ is an induced subgraph of $G$, then
$\reg(I(H)) \le \reg(I(G)).$

\item \cite[Remark~2.6]{selvi_ha} If $y \notin V(G)$, then
$\reg(I(G), y) = \reg(I(G)).$
\end{enumerate}
\end{theorem}

\begin{convention}
Throughout the paper, any expression for $\reg(I)$ involving $\left\lfloor \cdot \right\rfloor$ is taken to be $0$ whenever it is negative, since $\reg(I) \ge 0$.
\end{convention}

\section{Auxiliary Results}\label{AR}
In this section, we collect several auxiliary results that will be used in the proof of the main theorem. We begin with the following proposition, which describes the behavior of even-connections under disjoint unions.

\begin{proposition}\label{disjoint-even}
Let $G=\coprod_{i=1}^k G_i$ be a disjoint union of graphs, and let $M$ be a matching of $G$. Then
$G^M=\coprod_{i=1}^k (G_i)^{M_i},$
where $M_i := M \cap E(G_i)$ for each $1 \le i \le k$.
\end{proposition}

\begin{proof}
Write $M=\{e_1,\ldots,e_s\}$ and set $M_i := M \cap E(G_i)$ for each $i$. Then
$M=\coprod_{i=1}^k M_i$ and $\supp(M)=\bigcup_{i=1}^k \supp(M_i)$.
We first claim that for any $u,v \in V(G_i)$,
\[
u \sim_M v \text{ in } G
~ \Longleftrightarrow ~
u \sim_{M_i} v \text{ in } G_i.
\]
Suppose $u \sim_M v$ in $G$. Then there exists an even-connection
$p_0,\ldots,p_{2r+1}$ from $u$ to $v$. Since $G=\coprod_{i=1}^k G_i$, each edge $\{p_\ell,p_{\ell+1}\}$ lies in a unique component. As $u,v \in V(G_i)$, it follows that $p_\ell \in V(G_i)$ for all $\ell$, and hence the entire path lies in $G_i$. Thus $u \sim_{M_i} v$ in $G_i$.
Conversely, if $u \sim_{M_i} v$ in $G_i$, then the same sequence of vertices gives an even-connection in $G$ with respect to $M$, since $M_i \subseteq M$.

Now let $\{u,v\} \in E(G^M)$. Then either $\{u,v\} \in E(G)$ or $u \sim_M v$. In either case, $u$ and $v$ lie in the same component $G_i$. By the claim above, $\{u,v\} \in E((G_i)^{M_i})$, and hence
$E(G^M) \subseteq \bigcup_{i=1}^k E((G_i)^{M_i}).$
The reverse inclusion follows similarly: if $\{u,v\} \in E((G_i)^{M_i})$, then either $\{u,v\} \in E(G_i) \subseteq E(G)$ or $u \sim_{M_i} v$, which implies $u \sim_M v$. Thus $\{u,v\} \in E(G^M)$.
Therefore,
$G^M=\coprod_{i=1}^k (G_i)^{M_i},$
as claimed.
\end{proof}

The following proposition describes the behavior of even-connections under vertex deletion.

\begin{proposition}\label{del-even}
Let $G$ be a graph and $M$ be a matching of $G$. If $x \notin \supp(M)$ for some $x \in V(G)$, then
$G^M \setminus x = (G \setminus x)^M.$
\end{proposition}

\begin{proof}
We show that the two graphs have the same edge set. Let $u,v \in V(G)\setminus \{x\}$. Then
\[
\{u,v\} \in E(G^M \setminus x)
~ \Longleftrightarrow ~
\{u,v\} \in E(G^M).
\]
By definition, this holds if and only if either $\{u,v\} \in E(G)$ or $u \sim_M v$ in $G$.
Since $x \notin \supp(M)$, no edge of $M$ is incident to $x$. Hence, if there exists an even-connection between $u$ and $v$ with respect to $M$ in $G$, then there exists one that avoids $x$. Therefore,
$u \sim_M v \text{ in } G
~ \Longleftrightarrow ~
u \sim_M v \text{ in } G \setminus x.
$
It follows that
$\{u,v\} \in E(G^M \setminus x)
~ \Longleftrightarrow ~
\{u,v\} \in E((G \setminus x)^M).$
Thus,
$E(G^M \setminus x) = E((G \setminus x)^M),$
and hence the result follows.
\end{proof}

We fix notation and an ordering on the minimal generators of the squarefree powers of the edge ideal of a cycle. This ordering will be used throughout to compare generators and to control colon ideals arising in subsequent arguments.
\begin{notation}\label{setup-cycle}
Let $C_n$ be the cycle with vertex set $V(C_n)=\{x_1,\dots,x_n\}$ and edge set
$E(C_n)=\bigl\{\{x_1,x_2\}, \{x_2,x_3\}, \dots, \{x_{n-1},x_n\}, \{x_1,x_n\}\bigr\}.$
Fix $q \ge 1$. Then $I(C_n)^{[q-1]}$ is minimally generated by monomials corresponding to matchings of size $q-1$. 
For such a matching $M$, set
\[
m_M := \prod_{x \in \bigcup_{e \in M} e} x.
\]
We fix the following ordering on the edges of $C_n$:
\[
\{x_1,x_2\} < \{x_2,x_3\} < \cdots < \{x_{n-1},x_n\} < \{x_1,x_n\}.
\]
For a matching $M=\{e_1,\dots,e_{q-1}\}$, we arrange its edges in increasing order with respect to the above ordering and denote this sequence by $e(M)=(e^{(1)},\dots,e^{(q-1)})$.
This induces an order on the minimal generators as follows: for two matchings $M$ and $M'$, we declare $m_M < m_{M'}$ if there exists an index $i$ such that $e^{(j)}=e'^{(j)}$ for all $j<i$, and $e^{(i)} < e'^{(i)}$. If distinct matchings yield the same monomial, we choose the one whose sequence $e(M)$ is minimal with respect to this ordering. 
We write
\[
m_1 < \cdots < m_t
\]
for the resulting ordered list of minimal generators of $I(C_n)^{[q-1]}$.
\end{notation}

The following lemma controls the associated colon ideals under the above ordering.

\begin{lemma}\label{cycle-lemma}
Assume Notation~\ref{setup-cycle}, and let $m_1<\cdots<m_t$ be the minimal generators of $I(C_n)^{[q-1]}$. For $j<i$, either
$(m_j:m_i)\subseteq (I(C_n)^{[q]}:m_i),$
or there exists $r<i$ such that $(m_r:m_i)$ is generated by a variable and
$(m_j:m_i)\subseteq (m_r:m_i).$
\end{lemma}

\begin{proof}
Let $m_j,m_i\in \mathcal{G}(I(C_n)^{[q-1]})$ with $j<i$, and let $M_j,M_i$ be the corresponding matchings. Let $\{x_a,x_{a+1}\}$ be the smallest edge of $M_j$.
If $\{x_a,x_{a+1}\}$ is disjoint from $M_i$, then $M_i\cup\{\{x_a,x_{a+1}\}\}$ is a matching of size $q$. Hence
$(m_j:m_i)\subseteq (x_ax_{a+1})\subseteq (I(C_n)^{[q]}:m_i).$
Otherwise, $\{x_a,x_{a+1}\}$ is adjacent to an edge of $M_i$. Now we discuss the following cases:

\vskip 1mm
\noindent
\textit{Case-1:} Suppose first that $\{x_a,x_{a+1}\}$ is adjacent only with $\{x_{a+1},x_{a+2}\}\in M_i.$ Then there is a path $x_{a}x_{a+1}\dots x_{a+t}$ such that $\{x_{a+i},x_{a+i+1}\}\in M_j \text{ if } i \text{ is even and } \{x_{a+i},x_{a+i+1}\}\in M_i \text{ if } i \text{ is odd} $ and $\{x_{a+t},x_{a+t+1}\}\notin M_j,
M_i$. 
Suppose $\{x_{a+t-1},x_{a+t}\}\in M_i$
Then
$(m_j:m_i)\subseteq (x_a).$
Set
$M_r := (M_i\setminus\{\{x_{a+1},x_{a+2}\}\})\cup\{\{x_a,x_{a+1}\}\}.$
Then $M_r$ is a matching of size $q-1$, and by the ordering in Notation~\ref{setup-cycle}, we have $m_r<m_i$. Moreover,
$(m_r:m_i)=(x_a)\supseteq (m_j:m_i).$
Suppose $\{x_{a+t-1},x_{a+t}\}\in M_j$ then there is an even connection between $x_a \text{ and } x_{a+t}$ so $(m_j:m_i)\subseteq (I(C_n)^{[q]}:m_i).$

\vskip 1mm
\noindent
\textit{Case-2:}
 Suppose $\{x_a,x_{a+1}\}$ is adjacent with $\{x_{a+1},x_{a+2}\} \text{ and } \{x_{a-1},x_a\}\in M_i$. 
  Observe that in this case $\{x_a,x_{a+1}\}=\{x_1,x_2\} \text{ and } \{x_{a-1},x_a\}= \{x_n,x_1\}.$ 
 Then there is a path $P:=z_p\dots z_1x_nx_1x_2\dots x_t$ where edges are alternatively in $M_i \text{ and } M_j.$ Suppose $\{z_p,z_{p-1}\}\in M_i, \{x_{t-1},x_t\}\in M_j$. Then choose $M_r=M_i\setminus \{ \cup_{e\in P\cap M_i} e\}\bigcup_{e\in P\cap M_j}e.$ Then $M_r$ is a matching of size $q-1$. Also $m_r<m_i$ and $(m_r: m_i)=(x_t) \supseteq (m_j:m_i).$
 Now suppose $\{z_p,z_{p-1}\}\in 
M_j,\{x_{t-1},x_t\}\in M_i.$ Then similarly we can proceed. Suppose $\{z_p,z_{p-1}\}\in 
M_j,\{x_{t-1},x_t\}\in M_j.$ Then there exists an even connection so we are done. Suppose $\{z_p,z_{p-1}\}\in 
M_i,\{x_{t-1},x_t\}\in M_i.$ Since $|M_i|=|M_j|,$ either there exists an even connection or there exists an edge of $M_j$ which is disjoint from any other edges of $M_i$. In all 
cases we are done.
\end{proof}

 We now extend the above notation and ordering to the whiskered cycle.

\begin{notation}\label{setup-whisker}
Let $G = W(C_n)$ with $V(C_n)=\{x_1,\ldots,x_n\}$ and 
$V(G)=\{x_1,\ldots,x_n, y_1,\ldots,y_n\}$. 
For $0 \le j \le q-2$, write
$\mathcal{G}\big(I(C_n)^{[q-1-j]}\big)
=\{m_{j,1},\ldots,m_{j,\alpha_j}\},$
where each $m_{j,a}$ corresponds to a $(q-1-j)$-matching of $C_n$. 
For $j = q-1$, we set $\alpha_{q-1}=1$ and $m_{q-1,1}=1$.
By Lemma~\ref{cycle-lemma}, we fix an ordering
$m_{j,1}<\cdots<m_{j,\alpha_j}$
(for $0 \le j \le q-2$) such that for $1 \le a < b \le \alpha_j$,
$(m_{j,a}:m_{j,b}) \subseteq (I(C_n)^{[q-j]}:m_{j,b}),$
or there exists $r<b$ such that $(m_{j,r}:m_{j,b})$ is generated by a variable and
$(m_{j,a}:m_{j,b}) \subseteq (m_{j,r}:m_{j,b}).$

Let
$\mathcal W=\big\{\{x_i,y_i\}\mid 1 \le i \le n\big\}.$
For $0 \le j \le q-1$ and $1 \le a \le \alpha_j$, set
\[
\mathcal S(m_{j,a})
=
\big\{S \subseteq \mathcal W \mid |S|=j,\ 
\{x_i : \{x_i,y_i\} \in S\} \cap \supp(m_{j,a}) = \emptyset \big\}.
\]
For each $m_{j,a}$, we fix an arbitrary ordering on the set $\mathcal S(m_{j,a})$.
For $S \in \mathcal S(m_{j,a})$, define
\[
x_{M_{j,a,S}} := m_{j,a}\prod_{\{x_i,y_i\}\in S} x_i y_i,
\]
where $M_{j,a,S}$ is the corresponding $(q-1)$-matching of $G$. 
For $0 \le j \le q-1$, set
\[
\mathcal M^{q-1}_j
=
\{x_{M_{j,a,S}} \mid 1 \le a \le \alpha_j,\ S \in \mathcal S(m_{j,a})\},
\qquad
\mathcal M^{q-1}
=
\bigcup_{j=0}^{q-1} \mathcal M^{q-1}_j.
\]
Then $\mathcal M^{q-1} = \mathcal{G}(I(G)^{[q-1]})$.
We order $\mathcal M^{q-1}$ by
$\mathcal M^{q-1}_0 \prec \cdots \prec \mathcal M^{q-1}_{q-1},$
and within each $\mathcal M^{q-1}_j$ by declaring
$x_{M_{j,a,S}} \prec x_{M_{j,b,T}}$
if either $a<b$, or $a=b$ and $S$ precedes $T$ in the chosen ordering of $\mathcal S(m_{j,a})$.
Thus, we obtain an ordered list of minimal generators
\[
m_1 \prec m_2 \prec \cdots \prec m_\alpha.
\]
\end{notation}

The following lemma extends the previous ordering property to the whiskered cycle.
\begin{lemma}\label{order-whisker}
Assume Notation~\ref{setup-whisker}. Let 
$m_1 \prec m_2 \prec \cdots \prec m_\alpha$
be the minimal generators of $I(G)^{[q-1]}$, ordered as above. 
Fix $i$ and let $j<i$. Then either 
$(m_j:m_i)\subseteq (I(G)^{[q]}:m_i)$,
or there exists $r<i$ such that $(m_r:m_i)$ is generated by a variable and 
$(m_j:m_i)\subseteq (m_r:m_i)$. 
Consequently,
\[
\big((I(G)^{[q]},m_1,\ldots,m_{i-1}):m_i\big)
= (I(G)^{[q]}:m_i) + (u_1,\ldots,u_s)
\]
for some variables $u_1,\ldots,u_s$.
\end{lemma}

\begin{proof}
Let $m_j,m_i \in \mathcal{G}(I(G)^{[q-1]})$ with $j<i$, and let $M_j,M_i$ be the corresponding $(q-1)$-matchings of $G$. 
We argue according to the decomposition 
$\mathcal M^{q-1}=\bigcup_{k=0}^{q-1}\mathcal M^{q-1}_k$.

\medskip
\noindent\textit{Case 1.} Suppose $m_j \in \mathcal M^{q-1}_{q-1}$.

Then $M_j$ consists entirely of whisker edges. By the ordering in Notation~\ref{setup-whisker}, it follows that $m_i \in \mathcal M^{q-1}_{q-1}$ as well. Since $m_j \prec m_i$, there exists an edge $\{x_a,y_a\} \in M_j \setminus M_i$. Hence
$(m_j:m_i) \subseteq (x_a y_a).$
Moreover, $M_i \cup \{\{x_a,y_a\}\}$ is a matching of size $q$, and thus
$(x_a y_a) \subseteq (I(G)^{[q]}:m_i).$
It follows that $(m_j:m_i) \subseteq (I(G)^{[q]}:m_i)$.

\medskip
\noindent\textit{Case 2.} Suppose $m_j \notin \mathcal M^{q-1}_{q-1}$.

\medskip
\noindent\textit{Subcase 2.1.} Suppose $m_i \in \mathcal M^{q-1}_{q-1}$.

Then $M_i$ consists entirely of whiskers, while $M_j$ contains at least one edge of $C_n$. Suppose there exists an edge of $M_j\cap E(C_n)$ or an whisker of $M_j$ which is disjoint from $M_i$, then the argument is similar to \textit{Case~1.} Suppose this is not the case.
Let $\{x_a,x_b\}$ be the smallest edge of $M_j \cap E(C_n)$. Then $\{x_a,x_b\}$ is adjacent with some whisker of $M_i$. Suppose $\{x_a,x_b\}$ is adjacent to only $\{x_b,y_b\} \in M_i$, then 
$(m_j:m_i) \subseteq (x_a)$. Define
$M_r = \big(M_i \setminus \{\{x_b,y_b\}\}\big) \cup \{\{x_a,x_b\}\}.$
Then $m_r \prec m_i$ and 
$(m_r:m_i) = (x_a) \supseteq (m_j:m_i)$. 
If $\{x_a,x_b\}$ is adjacent to both $\{x_a,y_a\}$ and $\{x_b,y_b\}$ in $M_i$, then, since $|M_j|=|M_i|$, not all edges of $M_j$ can be adjacent to whiskers of $M_i$. Hence there exists an edge or whisker of $M_j$ that is disjoint from $M_i$, reducing to one of the previous cases.

\medskip
\noindent\textit{Subcase 2.2.} Suppose $m_i \notin \mathcal M^{q-1}_{q-1}$.

Then both $M_j$ and $M_i$ contain at least one edge of $C_n$. 
Observe that if $\{x_k,x_{k+1}\} \in M_j$ for $k=1,3,\ldots,n-2$ and $\{x_k,x_{k+1}\} \in M_i$ for $k=2,4,\ldots,n-1$, then $m_j=m_i$, contradicting $j<i$. Hence this situation does not occur.
Let $\{x_1,x_2\}$ be the smallest edge in $M_j \cap E(C_n)$. If $\{x_1,x_2\}$ is disjoint from $M_i$ then we are done. Otherwise there exists a maximal path $P:= z_p\dots z_1x_1\dots x_t.$ where alternatively edges are in $M_j \text{ and } M_i.$ Possibly $P$  can be $P:= x_1\dots x_t.$ Suppose $\{z_p,z_{p-1}\},\{x_{t-1},x_t\}\in M_j$ then there exists an even connection between $z_p \text{ and } x_t.$ So in this case  $(m_j:m_i) \subseteq (I(G)^{[q]}:m_i)$. Now suppose  $\{z_p,z_{p-1}\},\{x_{t-1},x_t\}\in M_i$. Since $|M_i|=|M_j|,$ either there exists an even connection or there exists $e\in M_j$ where e is disjoint from $M_i$. So in this case also we are done. Suppose, $\{z_p,z_{p-1}\}\in M_i \text{ and } \{x_{t-1},x_t\}\in M_j.$ Then $(m_j:m_i)\subseteq(x_t).$  First assume that $\{z_p,z_{p-1}\}\in M_i\cap E(C_n)$ and $\{x_{t-1},x_t\}$ is a whisker edge. Since $m_j\prec m_i$ then there exists a connected component induced by the vertices of $M_j \text{ and } M_i$, where number of edges in $M_j >$ number of edges in $M_i.$ So if there is a connected component containing single edge of $M_j$ which is disjoint from $M_i,$ then we are done, otherwise  there exists a maximal path such that $Q:= w_1\dots w_p,$ where edges are alternatively in $M_i$ and $M_j$ and either, $\{w_1,w_2\}$ is a whisker of $M_i$  and  $\{w_{p-1},w_p\}\in M_j\cap E(C_n)$ or $\{w_1,w_2\}\in E(C_n)\cap M_j$  and $\{w_{p-1},w_p\}$ is a whisker of $M_i.$ Then either $(m_j:m_i)\subseteq (w_1) \text{ or } (m_j:m_i)\subseteq (w_p)$. Then choose $M_r=M_i\setminus \{ \cup_{e\in Q\cap M_i} e\}\bigcup_{e\in Q\cap M_j}e.$ Then $m_r\prec m_i$ and $(m_r:m_i)\supseteq (m_j:m_i).$ In all other cases choose $M_r=M_i\setminus \{ \cup_{e\in P\cap M_i} e\}\bigcup_{e\in P\cap M_j}e.$ Then $m_r\prec m_i$ and $(m_r:m_i)=(x_t)\supseteq (m_j:m_i).$ Now suppose $\{z_p,z_{p-1}\}\in M_j \text{ and } \{x_{t-1},x_t\}\in M_i,$  then similrly we can proceed.

Now let $\{x_a,x_b\}$ be the smallest edge in $M_j\cap E(C_n)$, where $a\not=1.$
If $\{x_a,x_b\}$ is disjoint from $M_i$, then by similar argument as in \textit{Case~1}  we can proceed, otherwise there exists a maximal path $P:= z_1x_a\dots x_t$ where edges are alternatively in $M_j \text{ and } M_i.$ Possibly $P:= x_a\dots x_t$. Then by similar argument as in previous case we are done.
\end{proof}

The following lemma provides a key bound for paths, which will be used in the cycle case.

\begin{lemma}\label{tech-path}
Let $G = W(P_n)$ and let $2 \le q \le \left\lfloor \frac{n}{2} \right\rfloor + 1$. 
If $N$ is a $(q-1)$-matching of $P_n$, then
\[
\reg\big(I(G^{N})\big) \le 2 + \left\lfloor \frac{n - q}{2} \right\rfloor.
\]
\end{lemma}

\begin{proof}
Let $P_n$ be the path on vertices $\{x_1,\ldots,x_n\}$, and let 
$V(G)=\{x_1,\ldots,x_n,y_1,\ldots,y_n\}$. 
We proceed by induction on $n$.
If $n=2$, then $q=2$ and $N=\{\{x_1,x_2\}\}$. In this case, $G^N \cong P_2$. 
Hence, by \cite[Theorem~1]{froberg}, we have
$\reg(I(G^N)) = 2 = 2 + \left\lfloor \frac{n-2}{2}\right\rfloor.$

Suppose $n>2$ and let $N$ be a $(q-1)$-matching of $P_n$. 
Consider the subgraph of $P_n$ induced by the vertices of $N$. 
Each connected component of this subgraph is a path. 
Let $N'$ be a component of size $q'$, where $1 \le q' \le q-1$, and let $i$ be the smallest index such that $\{x_i,x_{i+1}\} \in N'$.
 Then
\[
N'=\big\{\{x_i,x_{i+1}\},\{x_{i+2},x_{i+3}\},\ldots,\{x_{i+2q'-2},x_{i+2q'-1}\}\big\}.
\]

\medskip
\noindent\textit{Case 1. $i=1$.} Then
$N'=\big\{\{x_1,x_2\},\{x_3,x_4\},\ldots,\{x_{2q'-1},x_{2q'}\}\big\}.$
We have
\[
G^N \setminus y_1 \cong W(P_\ell)^{\,N\setminus \{\{x_1,x_2\}\}},
~
G^N \setminus N_{G^N}[y_1] \cong W(P_{\ell'})^{\,N \setminus N'}
\]
up to isolated vertices,
where $P_\ell$ is the path on vertices $\{x_3,\ldots,x_n\}$ and $P_{\ell'}$ is the path on vertices $\{x_{2q'+2},\ldots,x_n\}$.
By the induction hypothesis on $n$, we obtain
\[
\reg\big(I(G^N \setminus y_1)\big)
\le 2+\left\lfloor \frac{(n-2)-(q-1)}{2} \right\rfloor
= 2+\left\lfloor \frac{n-q-1}{2} \right\rfloor,
\]
and
\[
\reg\big(I(G^N \setminus N_{G^N}[y_1])\big)+1
\le 2+\left\lfloor \frac{(n-(2q'+1))-(q-q')}{2} \right\rfloor +1
\le 2+\left\lfloor \frac{n-q}{2} \right\rfloor.
\]
Therefore, by \cite[Lemma~3.2]{huneke},
\[
\reg(I(G^N))
\le \max\Big\{
\reg\big(I(G^N \setminus y_1)\big),\,
\reg\big(I(G^N \setminus N_{G^N}[y_1])\big)+1
\Big\}
\le 2+\left\lfloor \frac{n-q}{2} \right\rfloor.
\]

\medskip
\noindent\textit{Case 2. $i \neq 1$.}
Then
$N'=\big\{\{x_i,x_{i+1}\},\{x_{i+2},x_{i+3}\},\ldots,\{x_{i+2q'-2},x_{i+2q'-1}\}\big\}.$
We have
\[
G^N \setminus x_{i-1} \cong W(P_{i-2}) \coprod W(P_{n-i+1})^{\,N},
~
G^N \setminus N_{G^N}[x_{i-1}] \cong W(P_{i-3}) \coprod W(P_{n-i-2q'})^{\,N \setminus N'}
\]
up to isolated vertices. 
By the induction hypothesis on $n$, we obtain
\[
\reg\big(I(G^N \setminus x_{i-1})\big)
\le \reg(I(W(P_{i-2}))) + \reg(I(W(P_{n-i+1})^{\,N})) - 1,
\]
and hence
\[
\reg\big(I(G^N \setminus x_{i-1})\big)
\le \left(2+\left\lfloor \frac{i-3}{2} \right\rfloor\right)
+ \left(2+\left\lfloor \frac{n-i+1-q}{2} \right\rfloor\right) -1
\le 2+\left\lfloor \frac{n-q}{2} \right\rfloor.
\]
Similarly,
\[
\reg\big(I(G^N \setminus N_{G^N}[x_{i-1}])\big)+1
\le \reg(I(W(P_{i-3}))) + \reg(I(W(P_{n-i-2q')})^{\,N \setminus N'})),
\]
and hence
\[
\reg\big(I(G^N \setminus N_{G^N}[x_{i-1}])\big)+1
\le \left(2+\left\lfloor \frac{i-4}{2} \right\rfloor\right)
+ \left(2+\left\lfloor \frac{n-i-2q'-(q-q')}{2} \right\rfloor\right)
\le 2+\left\lfloor \frac{n-q}{2} \right\rfloor.
\]
Therefore, by \cite[Lemma~3.2]{huneke},
\[
\reg(I(G^N))
\le \max\Big\{
\reg\big(I(G^N \setminus x_{i-1})\big),\,
\reg\big(I(G^N \setminus N_{G^N}[x_{i-1}])\big)+1
\Big\}
\le 2+\left\lfloor \frac{n-q}{2} \right\rfloor.
\]
\end{proof}

We record the following simple observation for later use.

\begin{observation}\label{path-ind}
Let $P_n$ be a path on the vertex set $\{x_1,\ldots,x_n\}$, and let $G = W(P_n)$ with 
$V(G) = \{x_1,\ldots,x_n, y_1,\ldots,y_n\}$. 
Let $M$ be a perfect matching of $P_n$. Then $n$ is even, and $G^M$ is a bipartite graph with bipartition 
$\{y_1,y_3,\ldots,y_{n-1}\}$ and $\{y_2,y_4,\ldots,y_n\}$. 
Moreover,
$E(G^M)
= \bigl\{\{y_i,y_j\} \mid i \text{ odd},\, j \text{ even},\, i<j\bigr\}.$
In particular, $G^M$ is co-chordal, and hence, by \cite[Theorem~1]{froberg},
$\reg(I(G^M)) = 2.$

Let $G = W(C_n)$ and let $M$ be a perfect matching of $C_n$. Then $n$ is even, and $G^M$ is a complete bipartite graph.
In particular, $G^M$ is co-chordal, and hence
$\reg(I(G^M)) = 2.$
\end{observation}

The following lemma establishes the corresponding bound for cycles.

\begin{lemma}\label{tech-cycle}
Let $G = W(C_n)$ and let $2 \le q \le \left\lfloor \frac{n}{2} \right\rfloor + 1$. 
If $N$ is a $(q-1)$-matching of $C_n$, then
\[
\reg\big(I(G^N)\big) \le 2 + \left\lfloor \frac{n - q - 1}{2} \right\rfloor.
\]
\end{lemma}

\begin{proof}
Let $C_n$ be the cycle on vertices $\{x_1,\ldots,x_n\}$ with edges $\{x_i,x_{i+1}\}$ (indices taken modulo $n$), and let 
$V(G)=\{x_1,\ldots,x_n,y_1,\ldots,y_n\}$, 
where $\{x_i,y_i\}$ are the whisker edges.
If $n=3$, then necessarily $|N|=1$. In this case, $G^N$ is a triangle with a whisker attached to one vertex, and hence is co-chordal. 
By \cite[Theorem~1]{froberg}, we have
$\reg(I(G^N)) = 2 = 2+\left\lfloor \frac{n-3}{2}\right\rfloor.$
Assume $n \ge 4$ and let $N$ be a $(q-1)$-matching of $C_n$. 
Consider the subgraph of $C_n$ induced by the vertices of $N$. Each connected component of this subgraph is either a path or a cycle.

\medskip
\noindent\emph{Case 1. The induced subgraph is a cycle.}
By Observation~\ref{path-ind} and \cite[Theorem~1]{froberg}, we have
$\reg(I(G^N)) = 2 \le 2+\left\lfloor \frac{n-q-1}{2}\right\rfloor.$

\medskip
\noindent\emph{Case 2. Every connected component of the induced subgraph on the vertices of $N$ is a path.}

Let $N'$ be a component of size $q'$, where $1 \le q' \le q-1$, and let $i$ be the smallest index such that $\{x_i,x_{i+1}\} \in N'$. Then
$N'=\big\{\{x_i,x_{i+1}\},\{x_{i+2},x_{i+3}\},\ldots,\{x_{i+2q'-2},x_{i+2q'-1}\}\big\}.$
Thus, $|N'| = q'$. Without loss of generality, we may assume that $i=2$. Then
$G^N\setminus x_1 \cong W(P_{n-1})^{\,N}.$
Therefore, by Lemma~\ref{tech-path},
\[
\reg\big(I(G^N \setminus x_1)\big)
\le 2+\left\lfloor \frac{n-1-q}{2}\right\rfloor.
\]

\medskip
\noindent\emph{Subcase 2.1.} Suppose $\{x_{n-1},x_n\} \notin N$. Then
$G^N \setminus N_{G^N}[x_1] \cong W(P_t)^{\,N_1}$
up to isolated vertices,
where $P_t$ is the path on vertices $\{x_{2q'+2},\ldots,x_{n-1}\}$,
$t = n - 2q' - 2$
and
$N_1 = N \setminus N'.$
Thus, $|N_1| = q - 1 - q'$. 
By Lemma~\ref{tech-path}, we obtain
\[
\reg\big(I(G^N \setminus N_{G^N}[x_1])\big)+1
\le 3+\left\lfloor \frac{n-2q'-2-(q-q')}{2} \right\rfloor.
\]
Since $q' \ge 1$, it follows that
$\reg\big(I(G^N \setminus N_{G^N}[x_1])\big)+1
\le 2+\left\lfloor \frac{n-q-1}{2} \right\rfloor. $

\medskip
\noindent\emph{Subcase 2.2.} Suppose $\{x_{n-1},x_n\} \in N$. 
Let
$N'' = \big\{\{x_{n}, x_{n-1}\}, \{x_{n-2}, x_{n-3}\}, \ldots, \{x_{n-(2q''-2)}, x_{n-(2q''-1)}\}\big\}$
be a connected component of the induced subgraph on the vertices of $N$. 
Then $|N''| = q''$. 
Observe that
$G^N \setminus N_{G^N}[x_1] \cong W(P_t)^{\,N_1}$
up to isolated vertices,
where
$N_1 = N \setminus (N' \cup N''),$
$t = n - 2(q' + q'') - 3,$
and
$|N_1| = (q - 1) - q' - q''.$
By Lemma~\ref{tech-path}, we obtain
\[
\reg\big(I(G^N \setminus N_{G^N}[x_1])\big)+1
\le 3 + \left\lfloor \frac{n - 2(q' + q'') - 3 - (q - q' - q'')}{2} \right\rfloor.
\]
Since $q' + q'' \ge 1$, it follows that
$\reg\big(I(G^N \setminus N_{G^N}[x_1])\big)+1
\le 2 + \left\lfloor \frac{n - q - 1}{2} \right\rfloor.$

In both subcases, we have
$\reg\big(I(G^N \setminus N_{G^N}[x_1])\big)+1
\le 2+\left\lfloor \frac{n-q-1}{2} \right\rfloor.$
Therefore, by \cite[Lemma~3.2]{huneke},
\[
\reg(I(G^N))
\le \max\Big\{
\reg\big(I(G^N\setminus x_{1})\big),\,
\reg\big(I(G^N\setminus N_{G^N}[x_{1}])\big)+1
\Big\}
\le 2+\left\lfloor \frac{n-q-1}{2}\right\rfloor.
\]
\end{proof}

\section{Main Result}\label{m}

In this section, we prove the main result of this paper. 
We begin by analyzing the structure of the graph $G^M$ associated to a minimal generator $m_\ell \in \mathcal{M}_k^{q-1}$ with $k \neq 0$, and introduce suitable subsets of vertices that will play a key role in bounding the regularity.

\begin{discussion}\label{case-as}
Let $G = W(C_n)$ and assume the notation of Notation~\ref{setup-whisker}. 
Suppose that $m_{\ell} \in \mathcal{M}_k^{q-1}$ with $k \neq 0$, and let
$M = B \cup \big\{\{x_{w_1},y_{w_1}\},\ldots,\{x_{w_k},y_{w_k}\}\big\}$
be the corresponding $(q-1)$-matching of $G$, where $B \subseteq E(C_n)$ is a matching.  
Then
$G \setminus \{x_{w_1},\ldots,x_{w_k}\}
= \coprod_{i=1}^{\gamma'} W(H_i),$
where each $H_i$ is a path. By Proposition~\ref{disjoint-even},
\begin{equation}\label{eq:GM-decomp}
G^M = G^B \setminus \{x_{w_1},\ldots,x_{w_k}\}
= \coprod_{i=1}^{\gamma'} W(H_i)^{B_i},
\end{equation}
where $B_i := B \cap E(H_i)$.
For each $1 \le i \le \gamma'$, set $n_i := |V(H_i)|$ and $b_i := |B_i|$, and write
$V(H_i) = \{x_{i,1}, x_{i,2}, \ldots, x_{i,n_i}\},$
so that $H_i$ is a path with edges $\{x_{i,j},x_{i,j+1}\} \in E(C_n)$ for $1 \le j \le n_i-1$. In particular, the endpoints of $H_i$ are $x_{i,1}$ and $x_{i,n_i}$.
Since each $H_i$ is a connected component of $C_n \setminus \{x_{w_1},\ldots,x_{w_k}\}$, there exist vertices $x_{w_{i'}}, x_{w_{i''}} \in \{x_{w_1},\ldots,x_{w_k}\}$ such that
$\{x_{w_{i'}},x_{i,1}\}$, $\{x_{i,n_i},x_{w_{i''}}\} \in E(C_n),$
where possibly $x_{w_{i'}} = x_{w_{i''}}$ if $\gamma'=1$. Without loss of generality, we assume $x_{w_{i'}} < x_{w_{i''}}$ with respect to the ordering in Notation~\ref{setup-whisker}.

After reindexing the decomposition in \eqref{eq:GM-decomp}, we may assume that
$0 \le \gamma_1' \le \gamma_2' \le \gamma'$
such that
\begin{enumerate}
\item $\supp(B_i) \cap \{x_{i,1},x_{i,n_i}\} = \emptyset$
for $1 \le i \le \gamma_1'$;

\item $\supp(B_i) \cap \{x_{i,1},x_{i,n_i}\} \neq \emptyset$
for $\gamma_1'+1 \le i \le \gamma'$, and
\begin{enumerate}
\item $x_{i,1} \in \supp(B_i)$
for $\gamma_1'+1 \le i \le \gamma_2'$;

\item $x_{i,1} \notin \supp(B_i)$
for $\gamma_2'+1 \le i \le \gamma'$.
\end{enumerate}
\end{enumerate}

In particular, for $\gamma_2'+1 \le i \le \gamma'$, one has $x_{i,n_i} \in \supp(B_i)$.

We now define subsets of $V(G)$ which will be used in the proof of the main theorem.
\begin{enumerate}
    \item For $1 \le i \le \gamma_1'$, set
$S_{H_i} = \{x_{i,1}, x_{i,n_i}\}.$

\item For $\gamma_1'+1 \le i \le \gamma_2'$, consider a longest alternating path in $H_i$ starting at $x_{i,1}$, say
$x_{i,1} x_{i,2} \cdots x_{i,2j_i},$
such that
$\{x_{i,2t-1}, x_{i,2t}\} \in B_i $
and
$\{x_{i,2t}, x_{i,2t+1}\} \notin B_i$
for  $1 \le t \le j_i-1,$
and the last edge $\{x_{i,2j_i-1}, x_{i,2j_i}\}$ belongs to $B_i$. We define
$S_{H_i} = \{y_{i,2}, y_{i,4}, \ldots, y_{i,2j_i}, x_{i,2j_i+1}\}.$

\item For $\gamma_2'+1 \le i \le \gamma'$, consider a longest alternating path in $H_i$ starting at $x_{i,n_i}$, say
$$x_{i,n_i} x_{i,n_i-1} \cdots x_{i,n_i-(2f_i-1)},$$
such that the edges alternate between belonging to $B_i$ and not belonging to $B_i$, and the last edge belongs to $B_i$. We define
$S_{H_i} = \{x_{i,1}, x_{i,n_i-2f_i}\}.$
\end{enumerate}
\end{discussion}

The following lemma provides a key tool for controlling the colon ideals that arise in our arguments.
\begin{lemma}\label{main-tech-lemma}
Assume the hypothesis and notation as in Discussion~\ref{case-as}. If $z \in S_{H_i}$ for some $1 \le i \le \gamma'$, then there exists a monomial $m_r \in \mathcal{G}(I(G)^{[q-1]})$ such that
$(m_r : m_\ell) = (z)$
and
$m_r \prec m_\ell.$
\end{lemma}

\begin{proof}
Let $M_\ell$ be the $(q-1)$-matching corresponding to the monomial $m_\ell$.

\vskip 1mm
\noindent
\textit{Case 1.} Suppose $1 \le i \le \gamma_1'$. Then
$\supp(B_i) \cap \{x_{i,1}, x_{i,n_i}\} = \emptyset.$
Hence $x_{i,1}$ and $x_{i,n_i}$ are not incident to any edge of $B_i$, and the edges $\{x_{w_{i'}},y_{w_{i'}}\}$ and $\{x_{w_{i''}},y_{w_{i''}}\}$ belong to $M_\ell$.
Define
$M_r := M_\ell \setminus \{\{x_{w_{i'}},y_{w_{i'}}\}\} \cup \{\{x_{w_{i'}}, x_{i,1}\}\},$
and let $m_r$ be the corresponding monomial. Then $M_r$ is a $(q-1)$-matching of $G$, and
$(m_r : m_\ell) = (x_{i,1}).$
Moreover, the number of whisker edges in $M_r$ is strictly less than that in $M_\ell$, and hence $m_r \prec m_\ell$.
Similarly, define
$M_{r'} := M_\ell \setminus \{\{x_{w_{i''}},y_{w_{i''}}\}\} \cup \{\{x_{w_{i''}}, x_{i,n_i}\}\},$
and let $m_{r'}$ be the corresponding monomial. Then $M_{r'}$ is a $(q-1)$-matching of $G$, and
$(m_{r'} : m_\ell) = (x_{i,n_i}),$
and $m_{r'} \prec m_\ell$.

\vskip 1mm
\noindent
\textit{Case 2.} Suppose $\gamma_1'+1 \le i \le \gamma_2'$. Then
$S_{H_i} = \{y_{i,2}, y_{i,4}, \ldots, y_{i,2j_i}, x_{i,2j_i+1}\}.$
Let
$x_{i,1}, x_{i,2}, \ldots, x_{i,2j_i}$
be a longest alternating path in $H_i$ starting at $x_{i,1}$.
First assume that $z = x_{i,2j_i+1}$. Define
\[
\begin{aligned}
M_r :=\;& M_\ell 
\setminus \Big( \{\{x_{w_{i'}},y_{w_{i'}}\}\} 
\cup \{\{x_{i,2t-1},x_{i,2t}\} \mid 1 \le t \le j_i\} \Big) \\
&\cup \Big( \{\{x_{w_{i'}},x_{i,1}\}\}
\cup \{\{x_{i,2t},x_{i,2t+1}\} \mid 1 \le t \le j_i\} \Big).
\end{aligned}
\]
Then $M_r$ is a $(q-1)$-matching of $G$. Let $m_r$ be the corresponding monomial. By construction,
$(m_r : m_\ell) = (x_{i,2j_i+1}).$
Moreover, the number of whisker edges in $M_r$ is strictly less than that in $M_\ell$, and hence $m_r \prec m_\ell$.
Next, let $z = y_{i,2s}$ for some $1 \le s \le j_i$. Define
\[
\begin{aligned}
M_r :=\;& M_\ell 
\setminus \Big( \{\{x_{w_{i'}},y_{w_{i'}}\}\} 
\cup \{\{x_{i,2t-1},x_{i,2t}\} \mid 1 \le t \le s\} \Big) \\
&\cup \Big( \{\{x_{w_{i'}},x_{i,1}\}\}
\cup \{\{x_{i,2t},x_{i,2t+1}\} \mid 1 \le t \le s-1\} 
\cup \{\{x_{i,2s},y_{i,2s}\}\} \Big).
\end{aligned}
\]
Then $M_r$ is a $(q-1)$-matching of $G$. Let $m_r$ be the corresponding monomial. Then
$(m_r : m_\ell) = (y_{i,2s}),$
and, since $\{x_{w_{i'}},x_{i,1}\} < \{x_{i,1},x_{i,2}\}$ under the ordering in Notation~\ref{setup-cycle}, we have $m_r \prec m_\ell$.

\vskip 1mm
\noindent
\textit{Case 3.} Suppose $\gamma_2'+1 \le i \le \gamma'$. Then
$S_{H_i} = \{x_{i,1}, x_{i,n_i-2f_i}\}.$
First assume that $z = x_{i,1}$. Since $x_{i,1} \notin \supp(B_i)$, the argument of Case~1 applies, and hence there exists a monomial $m_r$ such that
$(m_r : m_\ell) = (x_{i,1})$
and
$m_r \prec m_\ell.$
Next, let $z = x_{i,n_i-2f_i}$. Consider a longest alternating path in $H_i$ starting at $x_{i,n_i}$, say
$x_{i,n_i}, x_{i,n_i-1}, \ldots, x_{i,n_i-(2f_i-1)},$
such that
$\{x_{i,n_i-(2t)},x_{i,n_i-(2t+1)}\} \in B_i$
for  $0 \le t \le f_i-1.$ 
Define
\[
\begin{aligned}
M_r :=\;& M_\ell 
\setminus \Big( \{\{x_{w_{i''}},y_{w_{i''}}\}\}
\cup \{\{x_{i,n_i-(2t)},x_{i,n_i-(2t+1)}\} \mid 0 \le t \le f_i-1\} \Big) \\
&\cup \Big( \{\{x_{w_{i''}},x_{i,n_i}\}\}
\cup \{\{x_{i,n_i-(2t+1)},x_{i,n_i-(2t+2)}\} \mid 0 \le t \le f_i-2\} \\
&\qquad \cup \{\{x_{i,n_i-(2f_i-1)},x_{i,n_i-2f_i}\}\} \Big).
\end{aligned}
\]

Then $M_r$ is a $(q-1)$-matching of $G$. Let $m_r$ be the corresponding monomial. By construction,
$(m_r : m_\ell) = (x_{i,n_i-2f_i}).$
Moreover, the number of whisker edges in $M_r$ is strictly less than that in $M_\ell$, and hence $m_r \prec m_\ell$.
\end{proof}

We now state and prove the main result of this paper.
\begin{theorem}\label{main}
Let $G = W(C_n)$. Then, for $1 \le q \le n$,
\[
\reg\big(I(G)^{[q]}\big)
= 2q + \left\lfloor \frac{n - q - 1}{2} \right\rfloor.
\]
\end{theorem}

\begin{proof}
By \cite[Corollary~5.3]{DRS24}, we have
$2q+\left\lfloor \frac{n-q-1}{2}\right\rfloor
\le \reg\big(I(G)^{[q]}\big).$
It remains to prove the reverse inequality. We proceed by induction on $q$. 
The case $q=1$ follows from Theorem~\ref{reg-known}. Assume that $q \ge 2$.
Let $m_1 \prec m_2 \prec \cdots \prec m_\alpha$ be the minimal generators of $I(G)^{[q-1]}$, ordered as in Notation~\ref{setup-whisker}. 
By Lemma~\ref{order-whisker} and \cite[Lemma~3.2]{huneke}, we obtain
\begin{align*}
\reg\big(I(G)^{[q]}\big)
\le \max \Big\{ &
\reg\big((I(G)^{[q]},m_1,\ldots,m_{i-1}):m_i\big)+2(q-1) \;\text{for } 1 \le i \le \alpha, \\
&\reg\big(I(G)^{[q-1]}\big)
\Big\}.
\end{align*}
By the induction hypothesis,
$\reg\big(I(G)^{[q-1]}\big)
\le 2(q-1)+\left\lfloor \frac{n-q}{2}\right\rfloor
\le 2q+\left\lfloor \frac{n-q-1}{2}\right\rfloor.$
Thus, it suffices to show that, for all $1 \le \ell \le \alpha$,
$\reg\big((I(G)^{[q]},m_1,\ldots,m_{\ell-1}):m_\ell\big)
\le 2+\left\lfloor \frac{n-q-1}{2}\right\rfloor.$
Assume Notation~\ref{setup-whisker}.

\medskip
\noindent
\textit{Case 1.} Suppose $m_\ell \in \mathcal{M}_0^{q-1}$. 
Then, by Theorem~\ref{Thm-even}, Lemma~\ref{order-whisker}, and Lemma~\ref{tech-cycle},
\[
\reg\big((I(G)^{[q]},m_1,\ldots,m_{\ell-1}):m_\ell\big)
\le \reg\big(I(G)^{[q]}:m_\ell\big)
\le 2+\left\lfloor \frac{n-q-1}{2}\right\rfloor.
\]

\medskip
\noindent
\textit{Case 2.} Suppose $m_\ell \in \mathcal{M}_k^{q-1}$, where $0<k<q-1$. Let
$M = B \cup \big\{\{x_{w_1},y_{w_1}\},\ldots,\{x_{w_k},y_{w_k}\}\big\}$
be the corresponding $(q-1)$-matching of $G$, where $B \subseteq E(C_n)$ is a matching.  
Assume the notation as in Discussion~\ref{case-as}. By \eqref{eq:GM-decomp}, we have
$G^M = G^B \setminus \{x_{w_1},\ldots,x_{w_k}\}
= \coprod_{i=1}^{\gamma'} W(H_i)^{B_i},$
where $B_i := B \cap E(H_i)$.
Set
$S := \bigcup_{i=1}^{\gamma'} S_{H_i}.$
By Lemma~\ref{main-tech-lemma}, we have
\[
\reg\big((I(G)^{[q]},m_1,\ldots,m_{\ell-1}):m_\ell\big)
\le \reg\big(I(G^M \setminus S)\big).
\]
Moreover,
$G^M \setminus S
= \coprod_{i=1}^{\gamma'} \big(W(H_i)^{B_i} \setminus S_{H_i}\big).$
Note that it may happen that $W(H_i)^{B_i} \setminus S_{H_i}$ has no edges, i.e., it consists only of isolated vertices. In this case, its regularity is equal to $1$. Hence, without loss of generality, we may assume that each block contains at least one edge.
We now estimate the regularity of each component.

\begin{enumerate}
\item For $1 \le i \le \gamma_1'$, by Proposition~\ref{del-even}, we have
$W(H_i)^{B_i} \setminus S_{H_i}
= W(H_i \setminus S_{H_i})^{B_i}.$
Moreover, $H_i \setminus S_{H_i}$ is a path on $n_i-2$ vertices, and $|B_i|=b_i$. Hence, by Lemma~\ref{tech-path},
\[
\reg\big(I(W(H_i \setminus S_{H_i})^{B_i})\big)
\le 2+\left\lfloor \frac{(n_i-2)-(b_i+1)}{2} \right\rfloor.
\]

\item For $\gamma_1'+1 \le i \le \gamma_2'$, define $\widetilde{H_i}$ to be the path on the vertex set
$\{x_{i,2j_i+2}, \ldots, x_{i,n_i}\},$
and set $\widetilde{B_i} := B_i \cap E(\widetilde{H_i})$. Let $\widetilde{b_i} := |\widetilde{B_i}|$. 
Then
$W(H_i)^{B_i} \setminus S_{H_i} \cong W(\widetilde{H_i})^{\widetilde{B_i}}$
up to isolated vertices. Indeed, after removing the vertices
$y_{i,2}, y_{i,4}, \ldots, y_{i,2j_i},$
the vertices $y_{i,1}, y_{i,3}, \ldots, y_{i,2j_i-1}$ become isolated. 
Moreover,
$|V(\widetilde{H_i})| = n_i - (2j_i+1)$
and $b_i = \widetilde{b_i} + j_i.$
Therefore, by Lemma~\ref{tech-path},
\begin{align*}
    \reg\big(I(W(H_i)^{B_i} \setminus S_{H_i})\big)
= \reg\big(I(W(\widetilde{H_i})^{\widetilde{B_i}})\big) 
\leq &  2+\left\lfloor \frac{n_i - \widetilde{b_i} - 2j_i - 2}{2} \right\rfloor\\
\leq & 2+\left\lfloor \frac{n_i - b_i - j_i - 2}{2} \right\rfloor
\end{align*}

\item For $\gamma_2'+1 \le i \le \gamma'$, we have
$W(H_i)^{B_i} \setminus S_{H_i}
\cong W(H_{i,1})^{B_{i,1}} \coprod W(H_{i,2})^{B_{i,2}}$
up to isolated vertices, where $H_{i,1}$ is the path on the vertex set
$\{x_{i,2}, \ldots, x_{i,n_i-(2f_i+1)}\},$
and $H_{i,2}$ is the path on the vertex set
$\{x_{i,n_i-(2f_i-1)}, \ldots, x_{i,n_i}\}.$
Here $B_{i,j} := B_i \cap E(H_{i,j})$ for $j=1,2$, and $B_{i,2}$ is a perfect matching of $H_{i,2}$. 
Therefore, by Lemma \ref{tech-path} and Observation \ref{path-ind},
\begin{align*}
\reg\big(I(W(H_i)^{B_i} \setminus S_{H_i})\big)
&= \reg\big(I(W(H_{i,1})^{B_{i,1}})\big)
+ \reg\big(I(W(H_{i,2})^{B_{i,2}})\big) - 1 \\
&\le \left(2+\left\lfloor \frac{n_{i,1}-b_{i,1}-1}{2} \right\rfloor\right)
+ 2 - 1 
= 3+\left\lfloor \frac{n_{i,1}-b_{i,1}-1}{2} \right\rfloor,
\end{align*}
where $n_{i,1} := |V(H_{i,1})|$ and $b_{i,1} := |B_{i,1}|$.
\end{enumerate}
We now estimate the regularity of the edge ideal of $G^M \setminus S$. Since
$G^M \setminus S
= \coprod_{i=1}^{\gamma'} \big(W(H_i)^{B_i} \setminus S_{H_i}\big),$
it follows that
$\reg\big(I(G^M \setminus S)\big)
= \sum_{i=1}^{\gamma'} \reg\big(I(W(H_i)^{B_i}\setminus S_{H_i})\big)
- (\gamma'-1).$
By the estimates obtained above for each block, we have
\begin{align*}
\reg\big(I(G^M \setminus S)\big)
\le\; & \sum_{i=1}^{\gamma_1'} \left(2+\left\lfloor \frac{n_i-b_i-3}{2}\right\rfloor\right) 
+ \sum_{i=\gamma_1'+1}^{\gamma_2'} \left(2+\left\lfloor \frac{n_i-b_i-j_i-2}{2} \right\rfloor\right) \\
&\quad + \sum_{i=\gamma_2'+1}^{\gamma'} \left(3+\left\lfloor \frac{n_{i,1}-b_{i,1}-1}{2}\right\rfloor\right)
- (\gamma'-1).
\end{align*}

Rewriting, we obtain
\begin{align*}
\reg\big(I(G^M \setminus S)\big)
\le\; & \gamma_1' 
+ \sum_{i=1}^{\gamma_1'} \left\lfloor \frac{n_i - b_i - 1}{2} \right\rfloor
+ \sum_{i=\gamma_1'+1}^{\gamma_2'} \left\lfloor \frac{n_i - b_i - j_i}{2} \right\rfloor \\
&\quad + (\gamma_2' - \gamma_1')
+ 2(\gamma' - \gamma_2')
- (\gamma' - 1) \\
&\quad + \sum_{i=\gamma_2'+1}^{\gamma'} 
\left\lfloor \frac{n_{i,1} - b_{i,1} + 1}{2} \right\rfloor.
\end{align*}

Simplifying constants, we get
\begin{align*}
\reg\big(I(G^M \setminus S)\big)
\le\; & \sum_{i=1}^{\gamma_1'} \left\lfloor \frac{n_i - b_i - 1}{2} \right\rfloor
+ \sum_{i=\gamma_1'+1}^{\gamma_2'} \left\lfloor \frac{n_i - b_i - j_i}{2} \right\rfloor \\
&\quad + \sum_{i=\gamma_2'+1}^{\gamma'} 
\left\lfloor \frac{n_{i,1} - b_{i,1} + 1}{2} \right\rfloor
+ (\gamma' - \gamma_2') + 1.
\end{align*}

Using basic properties of the floor function, it follows that
\begin{align*}
\reg\big(I(G^M \setminus S)\big)
\le\; & \sum_{i=1}^{\gamma_1'} \left\lfloor \frac{n_i - b_i}{2} \right\rfloor
+ \sum_{i=\gamma_1'+1}^{\gamma_2'} \left\lfloor \frac{n_i - b_i}{2} \right\rfloor \\
&\quad + \sum_{i=\gamma_2'+1}^{\gamma'} 
\left\lfloor \frac{(n_{i,1} - b_{i,1}) + (n_{i,2} - b_{i,2})}{2} \right\rfloor
+ (\gamma' - \gamma_2') + 1,
\end{align*}
since $(n_{i,2}-b_{i,2}) \ge 1$.
Again,
\begin{align*}
&\reg\big(I(G^M \setminus S)\big)
\le\;  \\
&\left\lfloor 
\sum_{i=1}^{\gamma_1'} \frac{n_i - b_i}{2}
+ \sum_{i=\gamma_1'+1}^{\gamma_2'} \frac{n_i - b_i}{2}
+ \sum_{i=\gamma_2'+1}^{\gamma'} \frac{n_{i,1} + n_{i,2} - b_{i,1} - b_{i,2}}{2}
+ \frac{2(\gamma' - \gamma_2')}{2}
\right\rfloor + 1.
\end{align*}

Thus,
\begin{align*}
\reg\big(I(G^M \setminus S)\big)
\le\; & \left\lfloor \frac{1}{2} \left( 
\sum_{i=1}^{\gamma_1'} n_i
+ \sum_{i=\gamma_1'+1}^{\gamma_2'} n_i
+ \sum_{i=\gamma_2'+1}^{\gamma'} (n_{i,1} + n_{i,2} + 2)
\right) \right. \\
&\quad \left.
- \frac{1}{2} \left(
\sum_{i=1}^{\gamma_1'} b_i
+ \sum_{i=\gamma_1'+1}^{\gamma_2'} b_i
+ \sum_{i=\gamma_2'+1}^{\gamma'} (b_{i,1} + b_{i,2})
\right) - 1 \right\rfloor + 2.
\end{align*}

Hence,
\begin{align*}
\reg\big(I(G^M \setminus S)\big)
&\le \left\lfloor \frac{n}{2} - \frac{q-1}{2} - 1 \right\rfloor + 2 
= \left\lfloor \frac{n - q - 1}{2} \right\rfloor + 2.
\end{align*}

\medskip
\noindent
\textit{Case 3.} Suppose $m_\ell \in \mathcal{M}_{q-1}^{q-1}$, and let $M$ be the corresponding $(q-1)$-matching. Then
$M = \big\{\{x_{w_1},y_{w_1}\},\ldots,\{x_{w_{q-1}},y_{w_{q-1}}\}\big\},$
that is, $M$ consists entirely of whisker edges. Assume the notation as in Discussion~\ref{case-as}. Then
$G^M = G \setminus \{x_{w_1},\ldots,x_{w_{q-1}}\}
= \coprod_{i=1}^{\gamma_1'} W(H_i),$
where each $H_i$ is a path. Note that $B_i=\emptyset$ for all $1 \le i \le \gamma_1'$.
For each $1 \le i \le \gamma_1'$, set $n_i := |V(H_i)|$. Then
$n = \sum_{i=1}^{\gamma_1'} n_i + (q-1).$
Since $B_i=\emptyset$, we have
$\supp(B_i) \cap \{x_{i,1},x_{i,n_i}\} = \emptyset,$
and hence, as in Discussion~\ref{case-as},
$S_{H_i} = \{x_{i,1}, x_{i,n_i}\}
\quad \text{for all } 1 \le i \le \gamma_1'.$
Set
$S := \bigcup_{i=1}^{\gamma_1'} S_{H_i}.$
By Lemma~\ref{main-tech-lemma},
\[
\reg\big((I(G)^{[q]}, m_1,\ldots,m_{\ell-1}):m_\ell\big)
\le \reg\big(I(G^M \setminus S)\big).
\]
Moreover,
$G^M \setminus S
= \coprod_{i=1}^{\gamma_1'} W(H_i \setminus S_{H_i}).$
Hence,
$\reg\big(I(G^M \setminus S)\big)
= \sum_{i=1}^{\gamma_1'} \reg\big(I(W(H_i \setminus S_{H_i}))\big)
- (\gamma_1' - 1).$
Since $H_i \setminus S_{H_i}$ is a path on $n_i-2$ vertices, by Theorem~\ref{reg-known},
$\reg\big(I(W(H_i \setminus S_{H_i}))\big)
\le 2+\left\lfloor \frac{n_i-3}{2} \right\rfloor.$
Therefore,
\begin{align*}
\reg\big(I(G^M \setminus S)\big)
&\le \sum_{i=1}^{\gamma_1'} \left(2+\left\lfloor \frac{n_i-3}{2} \right\rfloor\right)
- (\gamma_1' - 1) = \sum_{i=1}^{\gamma_1'} \left\lfloor \frac{n_i-1}{2} \right\rfloor + 1.
\end{align*}
Using $\sum_{i=1}^{\gamma_1'} n_i = n-(q-1)$ and the inequality
$\sum_{i=1}^{\gamma_1'} \left\lfloor \frac{n_i-1}{2} \right\rfloor
\le \sum_{i=1}^{\gamma_1'} \left\lfloor \frac{n_i}{2} \right\rfloor \le \left\lfloor \sum_{i=1}^{\gamma_1'} \frac{n_i}{2} \right\rfloor,$
we obtain,

$\reg\big(I(G^M \setminus S)\big)
\le 2+\left\lfloor \frac{n-q-1}{2} \right\rfloor.$
\end{proof}

\vspace*{1mm}
\noindent
\textbf{Acknowledgments.}  
The last authors acknowledge support from the Science and Engineering Research Board (SERB) and the National Board for Higher Mathematics (NBHM).

\vspace*{1mm}
\noindent
\textbf{Data availability statement.}  
Data sharing is not applicable to this article as no datasets were generated or analyzed during the current study.

\vspace*{1mm}
\noindent
\textbf{Conflict of interest.}  
The authors declare that they have no known competing financial interests or personal relationships that could have appeared to influence the work reported in this paper.

\bibliographystyle{abbrv}
\bibliography{refs}

\end{document}